\documentclass[12pt]{amsart}

\usepackage{ucs}

\usepackage{amssymb}
\usepackage{amsthm}
\usepackage{amsmath}
\usepackage{latexsym}
\usepackage[cp1251]{inputenc}
\usepackage{graphicx}
\usepackage{wrapfig}
\usepackage{caption}
\usepackage{subcaption}
\usepackage{indentfirst}
\usepackage[left=2.6cm,right=2.6cm,top=3cm,bottom=3cm,bindingoffset=0cm]{geometry}
\usepackage{enumerate}

\def\O{\mathbf{O}}

\DeclareMathOperator{\aut}{Aut}

\DeclareMathOperator{\cay}{Cay}
\DeclareMathOperator{\cyc}{Cyc}

\DeclareMathOperator{\id}{id}

\DeclareMathOperator{\iso}{Iso}

\DeclareMathOperator{\orb}{Orb}

\DeclareMathOperator{\rk}{rk}

\DeclareMathOperator{\Span}{Span}

\DeclareMathOperator{\sym}{Sym}
\DeclareMathOperator{\rad}{rad}
\DeclareMathOperator{\Sup}{Sup}
\DeclareMathOperator{\DCI}{DCI}
\DeclareMathOperator{\CI}{CI}
\def\r{\mathrm{right}}

\makeatletter 
\def\@seccntformat#1{\csname the#1\endcsname. } 
\def\@biblabel#1{#1.}

\makeatother

\title{The Cayley isomorphism property for the group $C^5_2\times C_p$}

\author{Grigory Ryabov}
\address{Sobolev Institute of Mathematics, Novosibirsk, Russia}
\address{Novosibirsk State University, Novosibirsk, Russia}
\email{gric2ryabov@gmail.com}
\thanks{The work is supported by Mathematical Center in Akademgorodok under agreement No.~075-15-2019-1613 with the Ministry of Science and Higher Education of the Russian Federation}

\date{}

\newtheorem{prop}{Proposition}[section]
\newtheorem{theo}{Theorem}[section]

\newtheorem{lemm}[prop]{Lemma}

\newtheorem*{corl1}{Corollary}
\theoremstyle{definition}

\newtheorem*{rem1}{Remark 1}

\begin{document}

\vspace{\baselineskip}
\vspace{\baselineskip}

\vspace{\baselineskip}

\vspace{\baselineskip}

\begin{abstract}
A finite group $G$ is called a \emph{$\DCI$}-group if two Cayley digraphs over $G$ are isomorphic if and only if their connection sets are conjugate by a group automorphism. We prove that the group $C_2^5\times C_p$, where $p$ is a prime, is a $\DCI$-group if and only if $p\neq 2$. Together with the previously obtained results, this implies that a group $G$ of order $32p$, where $p$ is a prime, is a $\DCI$-group if and only if $p\neq 2$ and $G\cong C_2^5\times C_p$.
\\
\\
\textbf{Keywords}: Isomorphisms, $\DCI$-groups, Schur rings.
\\
\textbf{MSC}: 05C25, 05C60, 20B25.
\end{abstract}

\maketitle

\section{Introduction}

Let $G$ be a finite group and $S\subseteq G$. The \emph{Cayley digraph} $\cay(G,S)$ over $G$ with the connection set $S$ is defined to be the digraph with the vertex set $G$ and the arc set $\{(g,sg): g\in G, s\in S\}$. Two Cayley digraphs over $G$ are called \emph{Cayley isomorphic} if there exists an isomorphism between them which is also an automorphism of $G$. Clearly, two Cayley isomorphic Cayley digraphs are isomorphic. The converse statement is not true in general (see~\cite{AlPar, ET}). A subset $S\subseteq G$ is called a \emph{$\CI$-subset} if for each $T\subseteq G$ the Cayley digraphs $\cay(G,S)$ and $\cay(G,T)$ are isomorphic if and only if they are Cayley isomorphic. A finite group $G$ is called a \emph{$\DCI$-group} (\emph{$\CI$-group}, respectively) if each subset of $G$ (each inverse-closed subset of $G$, respectively) is a $\CI$-subset.

The investigation of $\DCI$-groups was initiated by \'Ad\'am~\cite{Adam} who conjectured, in our terms, that every cyclic group is a $\DCI$-group. This conjecture was disproved by Elspas and Turner in~\cite{ET}. The problem of determining of all finite $\DCI$- and $\CI$-groups was suggested by Babai and Frankl in~\cite{BF}. For more information on $\DCI$- and $\CI$-groups we refer the readers to the survey paper~\cite{Li}.

In this paper we are interested in abelian $\DCI$-groups. The cyclic group of order~$n$ is denoted by $C_n$. Elspas and Turner~\cite{ET} and independently Djokovi\'c~\cite{Dj} proved that every cyclic group of prime order is a $\DCI$-group. The fact that $C_{pq}$ is a $\DCI$-group for distinct primes $p$ and $q$ was proved by Alspach and Parsons in~\cite{AlPar} and independently by Klin and P\"oschel in~\cite{KP}. The complete classification of all cyclic $\DCI$-groups was obtained by Muzychuk in~\cite{M1,M2}. He proved that a cyclic group of order~$n$ is a $\DCI$-group if and only if $n=k$ or $n=2k$, where $k$ is square-free.

Denote the class of all finite abelian groups whose all Sylow subgroups are elementary abelian by $\mathcal{E}$. From~\cite[Theorem~1.1]{KM} it follows that every $\DCI$-group is the coprime product (i.e., the direct product of groups of coprime orders) of groups from the following list:
$$C_{p}^k,~C_4,~Q_8,~A_4,~H\rtimes \langle z \rangle,$$
where $p$ is a prime, $H$ is a group of odd order from $\mathcal{E}$, $|z|\in \{2,4\}$, and $h^z=h^{-1}$ for every $h\in H$. One can check that the class of $\DCI$-groups is closed under taking subgroups. So one of the  crucial steps towards the classification of all $\DCI$-groups is to determine which groups from $\mathcal{E}$ are $\DCI$. 

The following non-cyclic groups from $\mathcal{E}$ are $\DCI$-groups ($p$ and $q$ are assumed to be distinct primes): $C_p^2$~\cite{AlN,God}; $C_p^3$~\cite{AlN,Dob0}; $C_2^4$, $C_2^5$~\cite{CLi}; $C_p^4$, where $p$ is odd~\cite{HM} (a proof for $C_p^4$ with no condition on $p$ was given in~\cite{Mor}); $C_p^5$, where $p$ is odd~\cite{FK}; $C_p^2\times C_q$~\cite{KM}; $C_p^3\times C_q$~\cite{MS}; $C_p^4\times C_q$~\cite{KR2}. The smallest example of a non-$\DCI$-group from $\mathcal{E}$ was found by Nowitz~\cite{Now}. He proved that $C_2^6$ is non-$\DCI$. This implies that $C_2^n$ is non-$\DCI$ for every $n\geq 6$. Also $C_3^n$ is non-$\DCI$ for every $n\geq 8$~\cite{Sp} and $C_p^n$ is non-$\DCI$ for every prime $p$ and $n\geq 2p+3$~\cite{Som}.

In this paper we find a new infinite family of $\DCI$-groups from $\mathcal{E}$ which are close to the smallest non-$\DCI$-group from $\mathcal{E}$. The main result of the paper can be formulated as follows.

\begin{theo}\label{main}
Let $p$ be a prime. Then the group $C_2^5\times C_p$ is a $\DCI$-group if and only if $p\neq 2$.
\end{theo}

Theorem~\ref{main} extends the results obtained in~\cite{KM,KR2,MS} which imply that the group $C_p^k\times C_q$ is a $\DCI$-group whenever $p$ and $q$ are distinct primes and $k\leq 4$. Note that the ``only if'' part of Theorem~\ref{main}, in fact, was proved by Nowitz in~\cite{Now}. The next corollary immediately follows from~\cite[Theorem~1.1]{KM} and Theorem~\ref{main}.  

\begin{corl1}
Let $p$ be a prime. Then a group $G$ of order $32p$ is a $\DCI$-group if and only if $p\neq 2$ and $G\cong C_2^5\times C_p$.
\end{corl1}

To prove Theorem~\ref{main}, we use the $S$-ring approach. An \emph{$S$-ring} over a group $G$ is a subring of the group ring $\mathbb{Z}G$ which is a free $\mathbb{Z}$-module spanned by a special partition of $G$. If every $S$-ring from a certain family of $S$-rings over $G$ is a $\CI$-$S$-ring then $G$ is a $\DCI$-group (see Section~4). The definition of an $S$-ring goes back to Schur~\cite{Schur} and Wielandt~\cite{Wi}. The usage of $S$-rings in the investigation of $\DCI$-groups was proposed by Klin and P\"oschel~\cite{KP}. Many of recent results on $\DCI$-groups were obtained with using $S$-rings (see~\cite{HM,KM,KR1,KR2,MS}). 

The text of the paper is organized in the following way. In Section~2 we provide definitions and basic facts concerned with $S$-rings. Section~3 contains a necessary information on isomorphisms of $S$-rings. In Section~4 we discuss $\CI$-$S$-rings and their relation with $\DCI$-groups. Also in this section we prove a sufficient condition of $\CI$-property for $S$-rings (Lemma~\ref{minnorm}). Section~5 is devoted to the generalized wreath and star products of $S$-rings. Here we deduce from previously obtained results two sufficient conditions for the generalized wreath product of $S$-rings to be a $\CI$-$S$-ring (Lemma~\ref{kerci} and Lemma~\ref{citens}). Section~6 and~7 are concerned with $p$-$S$-rings and $S$-rings over groups of non-powerful order respectively. In Section~8 we provide properties of $S$-rings over the groups $C_2^n$, $n\leq 5$, and prove that all $S$-rings over these groups are $\CI$. The material of this section is based on computational results obtained with the help of the GAP package COCO2P~\cite{GAP}. Finally, in Section~9 we prove Theorem~\ref{main}.
\\
\\
\\
{\bf Notation.}

Let $G$ be a finite group and $X\subseteq G$. The element $\sum \limits_{x\in X} {x}$ of the group ring $\mathbb{Z}G$ is denoted by $\underline{X}$.

The set $\{x^{-1}:x\in X\}$ is denoted by $X^{-1}$.

The subgroup of $G$ generated by $X$ is denoted by $\langle X\rangle$; we also set $\rad(X)=\{g\in G:\ gX=Xg=X\}$.

%The order of $g\in G$ is denoted by $|g|$.

Given a set $X\subseteq G$ the set $\{(g,xg): x\in  X, g\in G\}$ of arcs of the Cayley digraph $\cay(G,X)$ is denoted by $R(X)$.

The group of all permutations of $G$ is denoted by $\sym(G)$.

The subgroup of $\sym(G)$ consisting of all right translations of $G$ is denoted by $G_{\r}$.

The set $\{K \leq \sym(G):~ K \geq G_{\r}\}$ is denoted by $\Sup(G_{\r})$.

For a set $\Delta\subseteq \sym(G)$ and a section $S=U/L$ of $G$ we set $\Delta^S=\{f^S:~f\in \Delta,~S^f=S\}$, where $S^f=S$ means that $f$ permutes the $L$-cosets in $U$ and $f^S$ denotes the bijection of $S$ induced by $f$.

If $K\leq \sym(\Omega)$ and $\alpha\in \Omega$ then the stabilizer of $\alpha$ in $K$ and the set of all orbits of $K$ on $\Omega$ are denoted by $K_{\alpha}$ and  $\orb(K,\Omega)$ respectively.

If $H\leq G$ then the normalizer of $H$ in $G$ is denoted by $N_G(H)$.

The cyclic group of order $n$ is denoted by  $C_n$.

The class of all finite abelian groups whose every Sylow subgroup is elementary abelian is denoted by $\mathcal{E}$.

\section{$S$-rings}

In this section we give a background of $S$-rings. In general, we follow~\cite{KR2}, where the most part of the material is contained. For more information on $S$-rings we refer the readers to~\cite{CP, MP1}.

Let $G$ be a finite group and $\mathbb{Z}G$ the integer group ring. Denote the identity element of $G$ by $e$. A subring  $\mathcal{A}\subseteq \mathbb{Z} G$ is called an \emph{$S$-ring (a Schur ring)} over $G$ if there exists a partition $\mathcal{S}(\mathcal{A})$ of~$G$ such that:

$(1)$ $\{e\}\in\mathcal{S}(\mathcal{A})$,

$(2)$  if $X\in\mathcal{S}(\mathcal{A})$ then $X^{-1}\in\mathcal{S}(\mathcal{A})$,

$(3)$ $\mathcal{A}=\Span_{\mathbb{Z}}\{\underline{X}:\ X\in\mathcal{S}(\mathcal{A})\}$.

\noindent The elements of $\mathcal{S}(\mathcal{A})$ are called the \emph{basic sets} of  $\mathcal{A}$ and the number $\rk(\mathcal{A})=|\mathcal{S}(\mathcal{A})|$ is called the \emph{rank} of  $\mathcal{A}$. If $X,Y \in \mathcal{S}(\mathcal{A})$ then $XY \in \mathcal{S}(\mathcal{A})$ whenever $|X|=1$ or $|Y|=1$.

Let $\mathcal{A}$ be an $S$-ring over a group $G$. A set $X \subseteq G$ is called an \emph{$\mathcal{A}$-set} if $\underline{X}\in \mathcal{A}$. A subgroup $H \leq G$ is called an \emph{$\mathcal{A}$-subgroup} if $H$ is an $\mathcal{A}$-set. From the definition it follows that the intersection of $\mathcal{A}$-subgroups is also an $\mathcal{A}$-subgroup. One can check that for each $\mathcal{A}$-set $X$ the groups  $\langle X \rangle$ and $\rad(X)$ are $\mathcal{A}$-subgroups. By the \emph{thin radical} of $\mathcal{A}$ we mean the set  defined as
$$\O_\theta(\mathcal{A})=\{ x \in G:~\{x\} \in \mathcal{S}(\mathcal{A}) \}.$$ 
It is easy to see that $\O_\theta(\mathcal{A})$ is an $\mathcal{A}$-subgroup.

\begin{lemm}~\cite[Lemma~2.1]{EKP}\label{intersection}
Let $\mathcal{A}$ be an $S$-ring over a group $G$, $H$ an $\mathcal{A}$-subgroup of $G$, and $X \in \mathcal{S}(\mathcal{A})$. Then the number $|X\cap Hx|$ does not depend on $x\in X$.
\end{lemm}

Let $L \unlhd U\leq G$. A section $U/L$ is called an \emph{$\mathcal{A}$-section} if $U$ and $L$ are $\mathcal{A}$-subgroups. If $S=U/L$ is an $\mathcal{A}$-section then the module
$$\mathcal{A}_S=Span_{\mathbb{Z}}\left\{\underline{X}^{\pi}:~X\in\mathcal{S}(\mathcal{A}),~X\subseteq U\right\},$$
where $\pi:U\rightarrow U/L$ is the canonical epimorphism, is an $S$-ring over $S$.

\section{Isomorphisms and schurity}

Let  $\mathcal{A}$  and $\mathcal{A}^{\prime}$ be $S$-rings over groups $G$  and $G^{\prime}$ respectively. A bijection $f:G\rightarrow G^{\prime}$ is called an \emph{isomorphism} from $\mathcal{A}$ to $\mathcal{A}^{\prime}$ if
$$\{R(X)^f: X\in \mathcal{S}(\mathcal{A})\}=\{R(X^{\prime}): X^{\prime}\in \mathcal{S}(\mathcal{A}^{\prime})\},$$
where $R(X)^f=\{(g^f,~h^f):~(g,~h)\in R(X)\}$. If there exists an isomorphism from $\mathcal{A}$ to $\mathcal{A}^{\prime}$ then we say that $\mathcal{A}$ and $\mathcal{A}^{\prime}$ are \emph{isomorphic} and write $\mathcal{A}\cong \mathcal{A}^{\prime}$. 

The group  of all isomorphisms from $\mathcal{A}$ onto itself contains a normal subgroup
$$\{f\in \sym(G): R(X)^f=R(X)~\text{for every}~X\in \mathcal{S}(\mathcal{A})\}$$
called the \emph{automorphism group} of $\mathcal{A}$ and denoted by $\aut(\mathcal{A})$. The definition implies that $G_{\r}\leq \aut(\mathcal{A})$. The $S$-ring $\mathcal{A}$ is called \emph{normal} if $G_{\r}$ is normal in $\aut(\mathcal{A})$. One can verify that if $S$ is an $\mathcal{A}$-section then $\aut(\mathcal{A})^S\leq \aut(\mathcal{A}_S)$. Denote the group $\aut(\mathcal{A})\cap \aut(G)$ by $\aut_G(\mathcal{A})$. It easy to check that if $S$ is an $\mathcal{A}$-section then $\aut_G(\mathcal{A})^S\leq \aut_S(\mathcal{A}_S)$. One can verify that
$$\aut_G(\mathcal{A})=N_{\aut(\mathcal{A})}(G_{\r})_e.$$

Let $K\in \Sup(G_{\r})$. Schur proved in~\cite{Schur} that the $\mathbb{Z}$-submodule
$$V(K,G)=\Span_{\mathbb{Z}}\{\underline{X}:~X\in \orb(K_e,~G)\},$$
is an $S$-ring over $G$. An $S$-ring $\mathcal{A}$ over  $G$ is called \emph{schurian} if $\mathcal{A}=V(K,G)$ for some $K\in \Sup(G_{\r})$. One can verify that given $K_1,K_2\in \Sup(G_{\r})$,
$$\text{if}~K_1\leq K_2~\text{then}~V(K_1,G)\geq V(K_2,G).~\eqno(1)$$
If $\mathcal{A}=V(K,G)$ for some $K\in \Sup(G_{\r})$ and $S$ is an $\mathcal{A}$-section then $\mathcal{A}_S=V(K^S,G)$. So if $\mathcal{A}$ is schurian then $\mathcal{A}_S$ is also schurian for every $\mathcal{A}$-section $S$. It can be checked that 
$$V(\aut(\mathcal{A}),G)\geq \mathcal{A}~\eqno(2)$$
and the equality is attained if and only if $\mathcal{A}$ is schurian.

An $S$-ring $\mathcal{A}$ over a group $G$ is defined to be \emph{cyclotomic} if there exists $K\leq\aut(G)$ such that $\mathcal{S}(\mathcal{A})=\orb(K,G)$. In this case we write $\mathcal{A}=\cyc(K,G)$. Obviously, $\mathcal{A}=V(G_{\r}K,G)$. So every cyclotomic $S$-ring is schurian. If $\mathcal{A}=\cyc(K,G)$ for some $K\leq \aut(G)$ and $S$ is an $\mathcal{A}$-section then $\mathcal{A}_S=\cyc(K^S,G)$. Therefore if $\mathcal{A}$ is cyclotomic then $\mathcal{A}_S$ is also cyclotomic for every $\mathcal{A}$-section $S$. 

Two permutation groups $K_1$ and $K_2$ on a set $\Omega$ are called \emph{$2$-equivalent} if $\orb(K_1,\Omega^2)=\orb(K_2,\Omega^2)$ (here we assume that $K_1$ and $K_2$ act on $\Omega^2$ componentwise). In this case we write $K_1\approx_2 K_2$. The relation $\approx_2$ is an equivalence relation on the set of all subgroups of $\sym(\Omega)$. Every equivalence class has a unique maximal element. Given $K\leq \sym(\Omega)$, this element is called the \emph{$2$-closure} of $K$ and denoted by $K^{(2)}$. If $\mathcal{A}=V(K,G)$ for some $K\in \Sup(G_{\r})$ then $K^{(2)}=\aut(\mathcal{A})$. An $S$-ring $\mathcal{A}$ over $G$ is called \emph{$2$-minimal} if
$$\{K\in \Sup(G_{\r}):~K\approx_2 \aut(\mathcal{A})\}=\{\aut(\mathcal{A})\}.$$

Two groups $K_1,K_2\leq \aut(G)$ are said to be \emph{Cayley equivalent} if $\orb(K_1,G)=\orb(K_2,G)$. In this case we write $K_1\approx_{\cay} K_2$. If $\mathcal{A}=\cyc(K,G)$ for some $K\leq \aut(G)$ then $\aut_G(\mathcal{A})$ is the largest group which is Cayley equivalent to $K$. A cyclotomic $S$-ring $\mathcal{A}$ over $G$ is called \emph{Cayley minimal} if
$$\{K\leq \aut(G):~K\approx_{\cay} \aut_G(\mathcal{A})\}=\{\aut_G(\mathcal{A})\}.$$
It is easy to see that $\mathbb{Z}G$ is Cayley minimal.

\section{$\CI$-$S$-rings}

Let $\mathcal{A}$ be an $S$-ring over a group $G$. Put 
$$\iso(\mathcal{A})=\{f\in \sym(G):~\text{f is an isomorphism from}~\mathcal{A}~\text{onto}~\text{$S$-ring over}~G\}.$$
One can see that $\aut(\mathcal{A})\aut(G)\subseteq \iso(\mathcal{A})$. However, the converse inclusion does not hold in general. The $S$-ring $\mathcal{A}$ is defined to be a \emph{$\CI$-$S$-ring} if $\aut(\mathcal{A})\aut(G)=\iso(\mathcal{A})$. It is easy to check that $\mathbb{Z}G$ and the $S$-ring of rank~$2$ over $G$ are $\CI$-$S$-rings.  

Put   
$$\Sup_2(G_\r)=\{K \in \Sup(G_\r):~ K^{(2)}=K\}.$$
The group $M\leq \sym(G)$ is said to be \emph{$G$-regular} if $M$ is regular and isomorphic to $G$. Following~\cite{HM}, we say that a group $K \in \Sup(G_\r)$ is \emph{$G$-transjugate} if every $G$-regular subgroup of $K$ is $K$-conjugate to $G_{\r}$. Babai proved in~\cite{Babai} the statement which can be formulated in our terms as follows: a set $S\subseteq G$ is a $\CI$-subset if and only if the group $\aut(\cay(G,S))$ is $G$-transjugate. The next lemma provides a similar criterion for a schurian $S$-ring to be $\CI$.

\begin{lemm}\label{trans}
Let $K \in \Sup_2(G_\r)$ and $\mathcal{A}=V(K,G)$. Then $\mathcal{A}$ is a $\CI$-$S$-ring if and only if  $K$ is $G$-transjugate. 
\end{lemm}

\begin{proof}
The statement of the lemma follows from~\cite[Theorem~2.6]{HM}.
\end{proof}

Let $K_1, K_2\in \Sup(G_\r)$ such that $K_1 \le K_2$. Then $K_1$ is called a \emph{$G$-complete  
subgroup} of $K_2$ if every $G$-regular subgroup of $K_2$ is $K_2$-conjugate to some $G$-regular subgroup of $K_1$ (see \cite[Definition~2]{HM}). In this case we write $K_1 \preceq_G K_2$. The relation $\preceq_G$ is a partial order on $\Sup(G_\r)$. The set of the minimal elements of $\Sup_2(G_\r)$ with respect to $\preceq_G$ is denoted by $\Sup_2^{\min}(G_\r)$.

\begin{lemm}\cite[Lemma~3.3]{KR2}\label{cimin}
Let $G$ be a finite group. If $V(K, G)$ is a $\CI$-$S$-ring for every $K \in \Sup_2^{\rm min}(G_\r)$ then $G$ is a $\DCI$-group.
\end{lemm}

\begin{rem1}
The condition that $V(K, G)$ is a $\CI$-$S$-ring for every $K \in \Sup_2^{\rm min}(G_\r)$ is equivalent to, say, that every schurian $S$-ring over $G$ is a $\CI$-$S$-ring. However, it is not known whether the statement converse to Lemma~\ref{cimin} is true.  
\end{rem1}

We finish the subsection with the lemma that gives a sufficient condition for an $S$-ring to be a $\CI$-$S$-ring. In order to formulate this condition, we need to introduce some further notations. Let $\mathcal{A}$ be a schurian $S$-ring over an abelian group $G$ and $L$ a normal $\mathcal{A}$-subgroup of $G$. Then the partition of $G$ into the $L$-cosets is $\aut(\mathcal{A})$-invariant. The kernel of the action of $\aut(\mathcal{A})$ on the latter cosets is denoted by $\aut(\mathcal{A})_{G/L}$. Since $\aut(\mathcal{A})_{G/L}$ is a normal subgroup of $\aut(\mathcal{A})$, we can form the group $K=\aut(\mathcal{A})_{G/L}G_{\r}$. Clearly, $K\leq\aut(\mathcal{A})$. From~\cite[Proposition~2.1]{HM} it follows that $K=K^{(2)}$.

\begin{lemm}\label{minnorm}
Let $\mathcal{A}$ be a schurian $S$-ring over an abelian group $G$, $L$ an $\mathcal{A}$-subgroup of $G$, and $K=\aut(\mathcal{A})_{G/L}G_{\r}$. Suppose that both $\mathcal{A}_{G/L}$ and $V(K, G)$ are $\CI$-$S$-rings and $\mathcal{A}_{G/L}$ is normal. Then $\mathcal{A}$ is a $\CI$-S-ring.
\end{lemm}

\begin{proof}
Firstly we prove that the group $\aut(\mathcal{A})^{G/L}$ is $G/L$-transjugate. Suppose that $F$ is a $G/L$-regular subgroup of $\aut(\mathcal{A})^{G/L}$. The $S$-ring  $\mathcal{A}_{G/L}$ is a $\CI$-$S$-ring by the assumption of the lemma. So Lemma~\ref{trans} implies that the group $\aut(\mathcal{A}_{G/L})$ is $G/L$-transjugate. Since $F\leq \aut(\mathcal{A})^{G/L}\leq \aut(\mathcal{A}_{G/L})$, we conclude that $F$ and $(G/L)_{\r}$ are $\aut(\mathcal{A}_{G/L})$-conjugate. However, $\mathcal{A}_{G/L}$ is normal and hence $F=(G/L)_{\r}$. Therefore $\aut(\mathcal{A})^{G/L}$ is $G/L$-transjugate. 

Now let us show that $K\preceq_G \aut(\mathcal{A})$. Let $H$ be a $G$-regular subgroup of $\aut(\mathcal{A})$. Then $H^{G/L}$ is abelian transitive subgroup of $\aut(\mathcal{A})^{G/L}$ and hence $H^{G/L}$ is regular on $G/L$. Therefore $H^{G/L}\cong (G/L)_{\r}=(G_{\r})^{G/L}$. There exists $\gamma\in \aut(\mathcal{A})$ such that $(H^{G/L})^{\gamma^{G/L}}=(G/L)_{\r}=(G_{right})^{G/L}$ because $\aut(\mathcal{A})^{G/L}$ is $G/L$-transjugate. This yields that $H^{\gamma}\leq K$. Thus, $K\preceq_G \aut(\mathcal{A})$. 

Finally, prove that $\aut(\mathcal{A})$ is $G$-transjugate. Again, let $H$ be a $G$-regular subgroup of $\aut(\mathcal{A})$. Since $K\preceq_G \aut(\mathcal{A})$, there exists $\gamma\in\aut(\mathcal{A})$ such that $H^{\gamma}\leq K$. The $S$-ring $V(K, G)$ is a $\CI$-$S$-ring by the assumption of the lemma. So $K$ is $G$-transjugate by Lemma~\ref{trans}. Therefore $H^{\gamma}$ and $G_{\r}$ are $K$-conjugate  and hence $H$ and $G_{\r}$ are $\aut(\mathcal{A})$-conjugate. Thus, $\aut(\mathcal{A})$ is $G$-transjugate and $\mathcal{A}$ is a $\CI$-$S$-ring by Lemma~\ref{trans}. The lemma is proved. 
\end{proof}

It should be mentioned that the proof of Lemma~\ref{minnorm} is quite similar to the proof of~\cite[Lemma~3.6]{KR2}.

\section{Generalized wreath and star products}

Let $\mathcal{A}$ be an $S$-ring over a group $G$ and $S=U/L$ an $\mathcal{A}$-section of $G$. An $S$-ring~$\mathcal{A}$ is called the \emph{$S$-wreath product} or the \emph{generalized wreath product} of $\mathcal{A}_U$ and $\mathcal{A}_{G/L}$ if $L\trianglelefteq G$ and $L\leq\rad(X)$ for each basic set $X$ outside~$U$. In this case we write $\mathcal{A}=\mathcal{A}_U\wr_{S}\mathcal{A}_{G/L}$ and omit $S$ when $U=L$. The construction of the generalized wreath product of $S$-rings was introduced  in~\cite{EP}. 

The $S$-wreath product is called \emph{nontrivial} or \emph{proper}  if $L\neq \{e\}$ and $U\neq G$. An $S$-ring $\mathcal{A}$ is said to be \emph{decomposable} if $\mathcal{A}$ is the nontrivial $S$-wreath product for some $\mathcal{A}$-section $S$ of $G$; otherwise $\mathcal{A}$ is said to be indecomposable. We say that an $\mathcal{A}$-subgroup $U<G$ has a \emph{gwr-complement} with respect to $\mathcal{A}$ if there exists a nontrivial normal $\mathcal{A}$-subgroup $L$ of $G$ such that $L\leq U$ and $\mathcal{A}$ is the $S$-wreath product, where $S=U/L$.

\begin{lemm}\cite[Theorem~1.1]{KR1}\label{cigwr}
Let $G\in \mathcal{E}$, $\mathcal{A}$ an $S$-ring over $G$, and $S=U/L$ an $\mathcal{A}$-section of $G$. Suppose that $\mathcal{A}$ is the nontrivial $S$-wreath product and the $S$-rings $\mathcal{A}_U$ and $\mathcal{A}_{G/L}$ are $\CI$-$S$-rings. Then $\mathcal{A}$ is a $CI$-$S$-ring whenever 
$$\aut_{S}(\mathcal{A}_{S})=\aut_U(\mathcal{A}_U)^{S}\aut_{G/L}(\mathcal{A}_{G/L})^{S}.$$ 
In particular,  $\mathcal{A}$ is a $\CI$-$S$-ring if $\aut_{S}(\mathcal{A}_{S})=\aut_U(\mathcal{A}_U)^{S}$ or $\aut_{S}(\mathcal{A}_{S})=\aut_{G/L}(\mathcal{A}_{G/L})^{S}$.
\end{lemm}

\begin{lemm}\cite[Proposition~4.1]{KR1}\label{trivial}
In the conditions of Lemma~\ref{cigwr}, suppose that $\mathcal{A}_S=\mathbb{Z}S$. Then $\mathcal{A}$ is a $\CI$-$S$-ring. In particular, if $U=L$ then $\mathcal{A}$ is a $\CI$-$S$-ring.
\end{lemm}

\begin{lemm}\cite[Lemma~4.2]{KR2}\label{cicayleymin}
In the conditions of Lemma~\ref{cigwr}, suppose that at least one of the $S$-rings $\mathcal{A}_U$ and $\mathcal{A}_{G/L}$ is cyclotomic and $\mathcal{A}_S$ is Cayley minimal. Then $\mathcal{A}$ is a $\CI$-$S$-ring.
\end{lemm}

\begin{lemm}\label{kerwreath}
Let $\mathcal{A}$ be an $S$-ring over an abelian group $G$. Suppose that $\mathcal{A}$ is the nontrivial $S=U/L$-wreath product for some $\mathcal{A}$-section $S=U/L$ and $L_1$ is an $\mathcal{A}$-subgroup containing $L$. Then $\mathcal{B}=V(K, G)$, where $K=\aut(\mathcal{A})_{G/L_1}G_{\r}$, is also the $S$-wreath product.
\end{lemm}

\begin{proof}
Since $K\leq \aut(\mathcal{A})$, from Eqs.~(1) and~(2) it follows that 
$$\mathcal{B}=V(K,G)\geq V(\aut(\mathcal{A}),G)\geq \mathcal{A}.$$ 
So $U$ and $L$ are also $\mathcal{B}$-subgroups. 

Let $\mathcal{C}=\mathbb{Z}U\wr_{S}\mathbb{Z}(G/L)$. The $S$-rings $\mathcal{C}_U$ and $\mathcal{C}_{G/L}$ are schurian and $\mathcal{C}_S$ is $2$-minimal. So $\mathcal{C}$ is schurian by~\cite[Corollary~10.3]{MP2}. This implies that 
$$\mathcal{C}=V(\aut(\mathcal{C}),G).~\eqno(3)$$ 

Every element from $\aut(\mathcal{C})_e$ fixes every basic set of $\mathcal{C}$ and hence it fixes every $L$-coset. Since $L_1\geq L$, every element from $\aut(\mathcal{C})_e$ fixes every $L_1$-coset. We conclude that $\aut(\mathcal{C})_e\leq \aut(\mathcal{A})_{G/L_1}$ and hence $\aut(\mathcal{C})\leq K$. Now from Eqs.~(1) and~(3) it follows that
$$\mathcal{C}=V(\aut(\mathcal{C}),G)\geq V(K,G)=\mathcal{B}.~\eqno(4)$$

The group $U$ is both $\mathcal{B}$-,$\mathcal{C}$-subgroup. Due to Eq.~(4), every basic set of  $\mathcal{B}$ which lies outside $U$ is a union of some basic sets of $\mathcal{C}$ which lie outside $U$. So $L\leq \rad(X)$ for every $X\in \mathcal{S}(\mathcal{B})$ outside $U$. Thus, $\mathcal{B}$ is the $S$-wreath product. The lemma is proved.
\end{proof}

\begin{lemm}\label{kerci}
In the conditions of Lemma~\ref{cigwr}, suppose that: $(1)$ every $S$-ring over $U$ is a $\CI$-$S$-ring; $(2)$ $\mathcal{A}_{G/L}$ is $2$-minimal or normal. Then $\mathcal{A}$ is a $\CI$-$S$-ring.
\end{lemm}

\begin{proof}
Let $\mathcal{B}=V(K,G)$, where $K=\aut(\mathcal{A})_{G/L}G_{\r}$. From Lemma~\ref{kerwreath} it follows that $\mathcal{B}$ is the $S$-wreath product. Since $L_1=L$, the definition of $\mathcal{B}$ implies that $\mathcal{B}_{G/L}=\mathbb{Z}(G/L)$ and hence $\mathcal{B}_S=\mathbb{Z}S$. Clearly, $\mathcal{B}_{G/L}$ is a $\CI$-$S$-ring. The $S$-ring $\mathcal{B}_U$ is a $\CI$-$S$-ring by the assumption of the lemma. Therefore $\mathcal{B}$ is a $\CI$-$S$-ring by Lemma~\ref{trivial}. The $S$-ring $\mathcal{A}_{G/L}$ is a $\CI$-$S$-ring by the assumption of the lemma. Thus, $\mathcal{A}$ is a $\CI$-$S$-ring by~\cite[Lemma~3.6]{KR2} whenever $\mathcal{A}_{G/L}$ is $2$-minimal and by Lemma~\ref{minnorm} whenever $\mathcal{A}_{G/L}$ is normal. The lemma is proved. 
\end{proof}

Let $V$ and $W$ be $\mathcal{A}$-subgroups. The $S$-ring $\mathcal{A}$ is called the \emph{star product} of $\mathcal{A}_V$ and $\mathcal{A}_W$  if the following conditions hold:

$(1)$ $V\cap W\trianglelefteq W$;

$(2)$ each $T\in \mathcal{S}(\mathcal{A})$ with $T\subseteq (W\setminus V) $ is a union of some $V\cap W$-cosets;

$(3)$  for each $T\in \mathcal{S}(\mathcal{A})$ with $T\subseteq G\setminus (V\cup W)$ there exist $R\in \mathcal{S}(\mathcal{A}_V)$ and $S\in \mathcal{S}(\mathcal{A}_W)$ such that $T=RS$.

\noindent In this case we write $\mathcal{A}=\mathcal{A}_V \star \mathcal{A}_W$. The construction of the star product of $S$-rings was introduced  in~\cite{HM}. The star product is called \emph{nontrivial} if $V\neq \{e\}$ and $V\neq G$. If $V\cap W=\{e\}$ then the star product is the usual \emph{tensor product} of $\mathcal{A}_V$ and $\mathcal{A}_W$ (see~\cite[p.5]{EKP}). In this case we write $\mathcal{A}=\mathcal{A}_V\otimes \mathcal{A}_W$. One can check that if $\mathcal{A}=\mathcal{A}_V\otimes \mathcal{A}_W$ then $\aut(\mathcal{A})=\aut(\mathcal{A}_V)\times \aut(\mathcal{A}_W)$. If $V\cap W\neq \{e\}$ then $\mathcal{A}$ is the nontrivial $V/(V\cap W)$-wreath product.

\begin{lemm}\label{cistar}
Let $G\in\mathcal{E}$ and  $\mathcal{A}$ a schurian $S$-ring over $G$. Suppose that $\mathcal{A}=\mathcal{A}_V \star \mathcal{A}_W$ for some $\mathcal{A}$-subgroups $V$ and $W$ of $G$ and the $S$-rings $\mathcal{A}_V$ and $\mathcal{A}_{W/(V\cap W)}$ are $\CI$-$S$-rings. Then $\mathcal{A}$ is a $\CI$-$S$-ring.
\end{lemm} 

\begin{proof}
The statement of the lemma follows from~\cite[Proposition~3.2, Theorem~4.1]{KM}.
\end{proof}

\begin{lemm}\cite[Lemma 2.8]{FK}\label{tenspr}
Let $\mathcal{A}$ be an $S$-ring over an abelian group $G=G_1\times G_2$. Assume that $G_1$ and $G_2$ are $\mathcal{A}$-groups. Then $\mathcal{A}=\mathcal{A}_{G_1}\otimes \mathcal{A}_{G_2}$ whenever $\mathcal{A}_{G_1}$ or $\mathcal{A}_{G_2}$ is the group ring.
\end{lemm}

\begin{lemm}\label{citens}
In the conditions of Lemma~\ref{cigwr}, suppose that $|G:U|$ is a prime and there exists $X\in \mathcal{S}(\mathcal{A}_{G/L})$ outside $S$ with $|X|=1$. Then $\mathcal{A}$ is a $\CI$-$S$-ring.
\end{lemm}

\begin{proof}
Let $X=\{x\}$ for some $x\in G/L$. Due to $G\in \mathcal{E}$, we conclude that $|\langle x \rangle|$ is prime. So $|\langle x \rangle \cap S|=1$ because $x$ lies outside $S$. Since $|G:U|$ is a prime, $G/L=\langle x \rangle \times S$. Note that $\mathcal{A}_{\langle x \rangle}=\mathbb{Z}\langle x \rangle$. Therefore 
$$\mathcal{A}_{G/L}=\mathbb{Z}\langle x \rangle \otimes \mathcal{A}_S$$ 
by Lemma~\ref{tenspr}. 

Let $\varphi \in \aut_{S}(\mathcal{A}_S)$. Define $\psi \in \aut(G/L)$ in the following way: 
$$\psi^{S}=\varphi,~x^\psi=x.$$ 
Then $\psi \in \aut_{G/L}(\mathcal{A}_{G/L})$ because $\mathcal{A}_{G/L}=\mathbb{Z}\langle x \rangle \otimes \mathcal{A}_S$. We obtain that $\aut_{G/L}(\mathcal{A}_{G/L})^S \geq \aut_S(\mathcal{A}_S)$, and hence $\aut_{G/L}(\mathcal{A}_{G/L})^S=\aut_S(\mathcal{A}_S)$. Thus, $\mathcal{A}$ is a $\CI$-$S$-ring by Lemma~\ref{cigwr}. The lemma is proved.
\end{proof}

\section{$p$-$S$-rings}

Let $p$ be a prime. An $S$-ring $\mathcal{A}$ over a $p$-group $G$ is called a \emph{$p$-$S$-ring} if every basic set of $\mathcal{A}$ has a $p$-power size. Clearly, if $|G|=p$ then $\mathcal{A}=\mathbb{Z}G$. In the next three lemmas $G$ is a $p$-group and $\mathcal{A}$ is a $p$-$S$-ring over $G$. 

\begin{lemm}\label{pextension}
If $\mathcal{B}\geq \mathcal{A}$ then $\mathcal{B}$ is a $p$-$S$-ring.
\end{lemm}

\begin{proof}
The statement of the lemma follows from~\cite[Theorem~1.1]{Pon}.
\end{proof}

\begin{lemm}\label{psection}
Let $S=U/L$ be an $\mathcal{A}$-section of $G$. Then $\mathcal{A}_S$ is a $p$-$S$-ring.
\end{lemm}

\begin{proof}
From Lemma~\ref{intersection} it follows that for every $X\in \mathcal{S}(\mathcal{A})$ the number $\lambda=|X \cap Lx|$ does not depend on $x\in X$. So $\lambda$ divides $|X|$ and hence $\lambda$ is a $p$-power. Let $\pi:G\rightarrow G/L$ be the canonical epimorphism. Note that $|\pi(X)|=|X|/\lambda$ and hence $|\pi(X)|$ is a $p$-power. Therefore every basic set of $\mathcal{A}_S$ has a $p$-power size. Thus, $\mathcal{A}_S$ is a $p$-$S$-ring. The lemma is proved. 
\end{proof}

\begin{lemm}\cite[Proposition~2.13]{FK}\label{psring}
The following statements hold:

$(1)$ $|\O_{\theta}(\mathcal{A})|>1$;

$(2)$ there exists a chain of $\mathcal{A}$-subgroups $\{e\}=G_0<G_1<\ldots G_s=G$ such that $|G_{i+1}:G_i|=p$ for every $i\in \{0,\ldots,s-1\}$.
\end{lemm}

\begin{lemm}\label{minpring}
Let $G$ be an abelian group, $K \in \Sup_2^{\min}(G_\r)$, and $\mathcal{A}=V(K, G)$. Suppose that $H$ is an $\mathcal{A}$-subgroup of $G$ such that $G/H$ is a $p$-group for some prime $p$. Then $\mathcal{A}_{G/H}$ is a $p$-S-ring.
\end{lemm}

\begin{proof}
The statement of the lemma follows from~\cite[Lemma~5.2]{KM}.
\end{proof}

\section{$S$-rings over over an abelian group of non-powerful order}

A number $n$ is called \emph{powerful} if $p^2$ divides $n$ for every prime divisor $p$ of $n$. From now throughout this subsection $G=H\times P$, where $H$ is an abelian group and $P\cong C_p$, where $p$ is a prime coprime to $|H|$. Clearly, $|G|$ is non-powerful. Let $\mathcal{A}$ be an $S$-ring over $G$, $H_1$ a maximal $\mathcal{A}$-subgroup contained in $H$, and $P_1$ the least $\mathcal{A}$-subgroup containing $P$. Note that $H_1P_1$ is an $\mathcal{A}$-subgroup.

%\begin{lemm}\cite[Proposition 13]{MS}\label{nonpower1}
%The group $H_1$ is a maximal $\mathcal{A}_{H_1P_1}$-subgroup.
%\end{lemm}

\begin{lemm}\cite[Lemma~6.3]{KR2}\label{nonpower2}
In the above notations, if $H_1 \neq (H_1P_1)_{p'}$, the Hall $p'$-subgroup of $H_1P_1$, then $\mathcal{A}_{H_1P_1}=\mathcal{A}_{H_1} \star \mathcal{A}_{P_1}$. 
\end{lemm}

\begin{lemm}\cite[Proposition 15]{MS}\label{nonpower3}
In the above notations, if $\mathcal{A}_{H_1P_1/H_1}\cong \mathbb{Z}C_p$ then $\mathcal{A}_{H_1P_1}=\mathcal{A}_{H_1} \star \mathcal{A}_{P_1}$. 
\end{lemm}

\begin{lemm}\cite[Lemma 6.2]{EKP}\label{nonpower4}
In the above notations, suppose that $H_1<H$. Then one of the following statements holds:

$(1)$ $\mathcal{A}=\mathcal{A}_{H_1}\wr \mathcal{A}_{G/H_1}$ with $\rk(\mathcal{A}_{G/H_1})=2$;

$(2)$ $\mathcal{A}=\mathcal{A}_{H_1P_1}\wr_S \mathcal{A}_{G/P_1}$, where $S=H_1P_1/P_1$ and $P_1<G$.
\end{lemm}

\section{$S$-rings over $C_2^n$, $n\leq 5$}

All $S$-rings over the groups $C_2^n$, where $n\leq 5$, were enumerated with the help of the GAP package COCO2P~\cite{GAP}. The lists of all $S$-rings over these  groups are available  on the web-page~\cite{Reich} (see also~\cite{Ziv}). The next lemma is an immediate consequence of the above computational results (see also~\cite[Theorem~1.2]{EKP}). 

\begin{lemm}\label{schurian}
Every $S$-ring over $C_2^n$, where $n\leq 5$, is schurian.
\end{lemm}

The main goal of this section is to describe $2$-$S$-rings over $C_2^n$, where $n\leq 5$, using computational results and to check that all $S$-rings over the above groups are $\CI$-$S$-rings. From now until the end of the section $G$ is an elementary abelian $2$-group of rank~$n$ and $\mathcal{A}$ is a $2$-$S$-ring over $G$.

\begin{lemm}\label{p3}
Let $n\leq 3$. Then $\mathcal{A}$ is cyclotomic. Moreover, $\mathcal{A}$ is Cayley minimal except for the case when $n=3$ and $\mathcal{A} \cong \mathbb{Z}C_2 \wr \mathbb{Z}C_2\wr \mathbb{Z}C_2$.
\end{lemm}

\begin{proof}
The first part of the lemma follows from~\cite[Lemma~5.2]{KR2}; the second part follows from~\cite[Lemma~5.3]{KR2}.
\end{proof}

Analyzing the lists of all $S$-rings over $C_2^4$ and $C_2^5$ available on the web-page~\cite{Reich}, we   conclude that up to isomorphism there are exactly nineteen $2$-$S$-rings over $G$ if $n=4$ and there are exactly one hundred $2$-$S$-rings over $G$ if $n=5$. It can can be established by inspecting the above $2$-$S$-rings one after the other that there are exactly fifteen decomposable and four indecomposable $2$-$S$-rings over $G$ if $n=4$ and there are exactly ninety six decomposable and four indecomposable $2$-$S$-rings over $G$ if $n=5$. 

\begin{lemm}\label{p45}
Let $n\in\{4,5\}$ and $\mathcal{A}$ indecomposable. Then $\mathcal{A}$ is normal. If in addition $n=5$ then $\mathcal{A}\cong \mathbb{Z}C_2\otimes \mathcal{A}^{\prime}$, where $\mathcal{A}^{\prime}$ is indecomposable $2$-$S$-ring over $C_2^4$. 
\end{lemm}

\begin{proof}
Let $n=4$. One can compute $|\aut(\mathcal{A})|$ and $|N_{\aut(\mathcal{A})}(G_{\r})|$ with using the GAP package COCO2P~\cite{GAP}. It turns out that for each of four indecomposable $2$-$S$-rings over $G$ the equality 
$$|\aut(\mathcal{A})|=|N_{\aut(\mathcal{A})}(G_{\r})|$$ 
is attained. So every indecomposable $2$-$S$-ring over $G$ is normal whenever $n=4$.

Let $n=5$. The straightforward check for each of four indecomposable $2$-$S$-rings over $G$ yields that $\mathcal{A}=\mathcal{A}_H\otimes \mathbb{Z}L$, where $H\cong C_2^4$, $L\cong C_2$, and $\mathcal{A}_H$ is indecomposable $2$-$S$-ring. Clearly, $\mathbb{Z}L$ is normal. By the above paragraph, $\mathcal{A}_H$ is normal. Since $\aut(\mathcal{A})=\aut(\mathcal{A}_H)\times \aut(\mathcal{A}_L)$, we obtain that $\mathcal{A}$ is normal. The lemma is proved.
\end{proof}

Note that if $p>2$ then Lemma~\ref{p45} does not hold. In fact, if $p>2$ then there exists an indecomposable $p$-$S$-ring over $C_p^5$ which is not normal (see~\cite[Lemma~6.4]{FK}).

\begin{lemm}\label{pnorm}
Let $n\leq 5$. Then $\mathcal{A}$ is normal whenever one of the following statements holds:

$(1)$ $\mathcal{A}$ is indecomposable;

$(2)$ $|G:\O_{\theta}(\mathcal{A})|=2$;

$(3)$ $n=4$ and $\mathcal{A}\cong (\mathbb{Z}C_2 \wr \mathbb{Z}C_2)\otimes (\mathbb{Z}C_2 \wr \mathbb{Z}C_2)$.
\end{lemm}

\begin{proof}
If $n\leq 3$ and $\mathcal{A}$ is indecomposable then $\mathcal{A}=\mathbb{Z}G$ by~\cite[Lemma~5.2]{KR2}. Clearly, in this case $\mathcal{A}$ is normal. If $n\in\{4,5\}$ and $\mathcal{A}$ is indecomposable then $\mathcal{A}$ is normal by Lemma~\ref{p45}. There are exactly~$n-1$ $2$-$S$-rings over $G$ for which Statement~2 of the lemma holds. For every $\mathcal{A}$ isomorphic to one of these $2$-$S$-rings and for $\mathcal{A}\cong (\mathbb{Z}C_2 \wr \mathbb{Z}C_2)\otimes (\mathbb{Z}C_2 \wr \mathbb{Z}C_2)$ one can compute $|\aut(\mathcal{A})|$ and $|N_{\aut(\mathcal{A})}(G_{\r})|$ with using the GAP package COCO2P~\cite{GAP}. It turns out that in each case the equality $|\aut(\mathcal{A})|=|N_{\aut(\mathcal{A})}(G_{\r})|$ holds and hence $\mathcal{A}$ is normal. The lemma is proved. 
\end{proof}

\begin{lemm}\label{p4cyclotom}
Let $n=4$. Then $\mathcal{A}$ is cyclotomic.
\end{lemm}

\begin{proof}
If $\mathcal{A}$ is decomposable then $\mathcal{A}$ is cyclotomic by~\cite[Lemma~5.6]{KR2}. If $\mathcal{A}$ is indecomposable then $\mathcal{A}$ is normal by Lemma~\ref{p45}. This implies that 
$$\aut(\mathcal{A})_e=N_{\aut(\mathcal{A})}(G_{\r})_e\leq \aut(G).$$ 
The $S$-ring $\mathcal{A}$ is schurian by Lemma~\ref{schurian}. So from Eq.~(2) it follows that $\mathcal{A}=V(\aut(\mathcal{A}),G)$ and hence $\mathcal{A}=\cyc(\aut(\mathcal{A})_e,G)$. The lemma is proved. 
\end{proof}

\begin{lemm}\label{p5cyclotom}
Let $n=5$. Suppose that $\mathcal{A}$ is decomposable and $|\O_{\theta}(\mathcal{A})|=8$. Then $\mathcal{A}$ is cyclotomic. 
\end{lemm}

\begin{proof}
Let $\mathcal{A}$ be the nontrivial $S$-wreath product for some $\mathcal{A}$-section $S=U/L$. Note that $|U|\leq 16$, $|G/L|\leq 16$, and $|S|\leq 8$. The $S$-rings $\mathcal{A}_{U}$, $\mathcal{A}_{G/L}$, and $\mathcal{A}_{S}$ are $2$-$S$-rings by Lemma~\ref{psection}. So each of these $S$-rings is cyclotomic by Lemma~\ref{p3} whenever the order of the corresponding group is at most~$8$ and by Lemma~\ref{p4cyclotom} otherwise. Since $|\O_\theta(\mathcal{A})|=8$, we conclude that $|S|\leq 4$ or $|S|=8$ and $|\O_\theta(\mathcal{A}_{S})|\geq 4$. In both cases $\mathcal{A}_{S}$ is Cayley minimal by Lemma~\ref{p3}. This implies that 
$$\aut_{U}(\mathcal{A}_{U})^{S}=\aut_{G/L}(\mathcal{A}_{G/L})^{S}=\aut_{S}(\mathcal{A}_{S}).$$
Now from~\cite[Lemma~4.3]{KR2} it follows that $\mathcal{A}$ is cyclotomic. The lemma is proved.
\end{proof}

\begin{lemm}\label{p5decomp1}
Let $n=5$. Suppose that $\mathcal{A}$ is decomposable and $|\O_\theta(\mathcal{A})|=4$. Then one of the following statements holds:

$(1)$ there exists an $\mathcal{A}$-subgroup $L\leq \O_\theta(\mathcal{A})$ of order~$2$ such that $\mathcal{A}=\mathbb{Z}\O_\theta(\mathcal{A})\wr_{S} \mathcal{A}_{G/L}$, where $S=\O_\theta(\mathcal{A})/L$;

$(2)$ $|\aut_G(\mathcal{A})|\geq |\aut_U(\mathcal{A}_U)|$ for every $\mathcal{A}$-subgroup $U$ with $|U|=16$ and $U\geq \O_\theta(\mathcal{A})$;

$(3)$ $\mathcal{A}$ is normal;

$(4)$ there exist an $\mathcal{A}$-subgroup $L\leq \O_\theta(\mathcal{A})$ and $X\in \mathcal{S}(\mathcal{A})$ such that $|L|=|X|=2$, $L\neq \rad(X)$, and $\mathcal{A}_{G/L}$ is normal.
\end{lemm}

\begin{proof}
There are exactly forty five decomposable $2$-$S$-rings over $G$ whose thin radical has order~$4$. This can be checked by inspecting all $2$-$S$-rings over $G$ one after the other. Let $\mathcal{A}$ be one of them and $R=\O_\theta(\mathcal{A})$. The straightforward check of basic sets of each of the above forty five $2$-$S$-rings shows that  Statement~1 of the lemma holds for twenty six of them. The analysis of basic sets of the remaining nineteen $2$-$S$-rings implies that  ten of them have an $\mathcal{A}$-subgroup $L\leq R$ and $X\in \mathcal{S}(\mathcal{A})$ satisfying the following: (1) $|L|=|X|=2$; (2) $L\neq \rad(X)$; (3) one of the conditions from Lemma~\ref{pnorm} holds for $\mathcal{A}_{G/L}$. We conclude that $\mathcal{A}_{G/L}$ is normal and hence Statement~4 of the lemma holds for these ten $2$-$S$-rings.

It remains to consider nine $2$-$S$-rings for which neither Statement~1 nor Statement~4 of the lemma  holds. Let $U$ be an $\mathcal{A}$-subgroup with $|U|=16$ and $U\geq R$. Statement~2 of Lemma~\ref{psring} yields that there exists an $\mathcal{A}_U$-subgroup $U_1$ such that
$$|U_1|=8~\text{and}~R<U_1<U.$$

Let $X_1$ be a basic set of $\mathcal{A}$ inside $U_1\setminus R$ of the least possible size and $X_2$ a basic set of $\mathcal{A}$ inside $U\setminus U_1$ of the least possible size. Clearly, $|X_1|\leq 4$ and $|X_2|\leq 8$. Choose $x_1\in X_1$ and $x_2\in X_2$. From the choice of $x_1$ and $x_2$ it follows that $\langle R, x_1, x_2\rangle=U$. So if $\varphi\in \aut_U(\mathcal{A}_U)$, $x_1^{\varphi}=x_1$, and $x_2^{\varphi}=x_2$ then $\varphi$ is trivial. This implies that
$$|\aut_U(\mathcal{A}_U)|\leq |X_1||X_2|\leq 32.~\eqno(5)$$
One can compute $|\aut_G(\mathcal{A})|=|N_{\aut(\mathcal{A})}(G_{\r})_e|$ with using the GAP package COCO2P~\cite{GAP}. The inequality $|\aut_G(\mathcal{A})|\geq 32$ holds for four of the remaining $2$-$S$-rings. Due to Eq.~(5), Statement~2 of the lemma holds for them. For three of the remaining $2$-$S$-rings, we have $|\aut_G(\mathcal{A})|=16$. However, in this situation there are no basic sets of size~$4$ and hence $|X_1|=2$ or there are no basic sets of size $8$ and hence $|X_2|\leq 4$. In both case Eq.~(5) implies that $|\aut_U(\mathcal{A}_U)|\leq 16$ and Statement~2 of the lemma holds. 

Now it remains to consider two $2$-$S$-rings for which $|\aut_G(\mathcal{A})|=8$. One of these $2$-$S$-rings does not have a basic set of size~$8$ and every of its $\mathcal{A}$-subgroups of order~$16$ contains a basic set of size~$2$. So $|X_1||X_2|\leq 8$. In view of Eq.~(5), Statement~2 of the lemma holds for this $2$-$S$-ring. For the other of these $2$-$S$-rings computer calculations made with the help the GAP package COCO2P~\cite{GAP} imply that $|\aut(\mathcal{A})_e|=|\aut_G(\mathcal{A})|=8$ and hence it is normal, i.e. Statement~3 of the lemma holds for it. The lemma is proved.
\end{proof}

\begin{lemm}\label{p5decomp2}
Let $n=5$. Suppose that $\mathcal{A}$ is decomposable, $|\O_\theta(\mathcal{A})|=2$, and there exists $X\in \mathcal{S}(\mathcal{A})$ with $|X|>1$ and $|\rad(X)|=1$. Then $|X|=4$ and one of the following statements holds:

$(1)$ $\mathcal{A}\cong \mathcal{B} \wr \mathbb{Z}C_2$, where $\mathcal{B}$ is a $2$-$S$-ring over $C_2^4$;

$(2)$ $|\aut_G(\mathcal{A})|\geq |\aut_U(\mathcal{A}_U)|$ for every $\mathcal{A}$-subgroup $U$ with $|U|=16$;

$(3)$ there exists an $\mathcal{A}$-subgroup $L$ such that $|L|\in\{2,4\}$ and $\mathcal{A}_{G/L}$ is normal.
\end{lemm}

\begin{proof}
There are exactly twenty nine decomposable $2$-$S$-rings over $G$ whose thin radical has order~$2$ (in fact, every $2$-$S$-ring with the thin radical of order~$2$ is decomposable). This can be verified by inspecting all $2$-$S$-rings over $G$ one after the other. Only ten of these twenty nine $2$-$S$-rings have a basic set with the trivial radical and each of such basic sets with the trivial radical has size~$4$. Let $\mathcal{A}$ be one of the ten above $2$-$S$-rings. From the direct check it follows that Statement~1 of the lemma holds for two of these ten $2$-$S$-rings. The analysis of basic sets of the remaining eight $2$-$S$-rings yields that six of them have an $\mathcal{A}$-subgroup $L$  satisfying the following: (1) $|L|\in\{2,4\}$; (2) one of the conditions from Lemma~\ref{pnorm} holds for $\mathcal{A}_{G/L}$. We conclude $\mathcal{A}_{G/L}$ is normal and hence Statement~3 of the lemma holds for these six $2$-$S$-rings.

It remains to consider two $2$-$S$-rings for which neither Statement~1 nor Statement~3 of the lemma  holds. Each of these two $2$-$S$-rings has exactly three distinct $\mathcal{A}$-subgroups of order~$16$, say $U_1$, $U_2$, and $U_3$. Computations made by using the GAP package COCO2P yield that for one of them the following holds:
$$|\aut_G(\mathcal{A})|=64,~|\aut_{U_i}(\mathcal{A}_{U_i})|\in \{8,64\}~\text{for}~i\in\{1,2,3\};$$
and for the other of them the following holds:
$$|\aut_G(\mathcal{A})|=32,~|\aut_{U_i}(\mathcal{A}_{U_i})|\in \{4,32\}~\text{for}~i\in\{1,2,3\}.$$
In both cases Statement~2 of the lemma holds. The lemma is proved.
\end{proof}

\begin{lemm}\label{2cigwr}
Let $D\in \mathcal{E}$ such that every $S$-ring over a proper section of $D$ is $\CI$, $\mathcal{D}$ an $S$-ring over $D$, and $S=U/L$ a $\mathcal{D}$-section. Suppose that $\mathcal{D}$ is the nontrivial $S$-wreath product. Then $\mathcal{D}$ is a $\CI$-$S$-ring whenever $D/L\cong C_2^k$ for some $k\leq 4$ and $\mathcal{D}_{D/L}$ is a $2$-$S$-ring.
\end{lemm}

\begin{proof}
The $S$-ring $\mathcal{D}_{D/L}$ is cyclotomic by Lemma~\ref{p3} whenever $|D/L|\leq 8$ and by Lemma~\ref{p4cyclotom} whenever $|D/L|=16$. The $S$-ring $\mathcal{D}_S$ is a $2$-$S$-ring by Lemma~\ref{psection}. If $\mathcal{D}_S\ncong \mathbb{Z}C_2\wr \mathbb{Z}C_2\wr \mathbb{Z}C_2$ then $\mathcal{D}_S$ is Cayley minimal by Lemma~\ref{p3}. The $S$-rings $\mathcal{D}_U$ and $\mathcal{D}_{D/L}$ are $\CI$-$S$-rings by the assumption of the lemma. So $\mathcal{D}$ is a $\CI$-$S$-ring by Lemma~\ref{cicayleymin}. 

Assume that 
$$\mathcal{D}_S\cong \mathbb{Z}C_2\wr \mathbb{Z}C_2\wr \mathbb{Z}C_2.$$
In this case $|D/L|=16$, $|S|=8$, and there exists the least $\mathcal{D}_S$-subgroup $A$ of $S$ of order~$2$. Every basic set of $\mathcal{D}_{D/L}$ outside $S$ is contained in an $S$-coset because  $\mathcal{D}_{(D/L)/S}\cong \mathbb{Z}C_2$. So $\rad(X)$ is an $\mathcal{D}_S$-subgroup for every $X\in \mathcal{S}(\mathcal{D}_{D/L})$ outside $S$. If $|\rad(X)|>1$ for every $X\in \mathcal{S}(\mathcal{D}_{D/L})$ outside $S$ then $\mathcal{D}_{D/L}$ is the $S/A$-wreath product because $A$ is the least $\mathcal{D}_S$-subgroup. This implies that $\mathcal{D}$ is the $U/\pi^{-1}(A)$-wreath product, where $\pi:D\rightarrow D/L$ is the canonical epimorphism. One can see that $|D/\pi^{-1}(A)|\leq 8$ and $|U/\pi^{-1}(A)|\leq 4$. The $S$-rings $\mathcal{D}_{D/\pi^{-1}(A)}$ and $\mathcal{D}_{U/\pi^{-1}(A)}$ are $2$-$S$-rings by Lemma~\ref{psection}. The $S$-ring $\mathcal{D}_{D/\pi^{-1}(A)}$ is cyclotomic by Lemma~\ref{p3} and the $S$-ring $\mathcal{D}_{U/\pi^{-1}(A)}$ is Cayley minimal by Lemma~\ref{p3}. The $S$-rings $\mathcal{D}_U$ and $\mathcal{D}_{D/\pi^{-1}(A)}$ are $\CI$-$S$-rings by the assumption of the lemma. Thus, $\mathcal{D}$ is a $\CI$-$S$-ring by Lemma~\ref{cicayleymin}.

Suppose that there exists a basic set $X$ of $\mathcal{D}_{D/L}$ outside $S$ with $|\rad(X)|=1$. If $\mathcal{D}_{D/L}$ is decomposable then 
$$\aut_{D/L}(\mathcal{D}_{D/L})^S=\aut_S(\mathcal{D}_S)$$
by~\cite[Lemma~5.8]{KR2}. Therefore $\mathcal{D}$ is a $\CI$-$S$-ring by Lemma~\ref{cigwr}.

If $\mathcal{D}_{D/L}$ is indecomposable then $\mathcal{D}_{D/L}$ is normal by Lemma~\ref{p45}. So all conditions of Lemma~\ref{kerci} hold for $\mathcal{D}$. Thus, $\mathcal{D}$ is a $\CI$-$S$-ring. The lemma is proved.
\end{proof}

From the results obtained in~\cite{AlN, CLi} it follows that the group $C_2^n$ is a $\DCI$-group for $n\leq 5$. However, this does not imply that every $S$-ring over $C_2^n$, where $n\leq 5$, is a $\CI$-$S$-ring (see Remark~1). Below we check that all $S$-rings over the above groups are $\CI$-$S$-rings.

\begin{lemm}\label{everyci}
Let $n\leq 5$. Then every $S$-ring over $G$ is a $\CI$-$S$-ring. 
\end{lemm}

\begin{proof}
Every $S$-ring over $G$ is schurian by Lemma~\ref{schurian}. So to prove the lemma, it is sufficient to prove that $\mathcal{B}=V(K,G)$ is a $\CI$-$S$-ring for every $K\in\Sup_2^{\rm min}(G_\r)$ (see Remark~1). The $S$-ring $\mathcal{B}$ is a $2$-$S$-ring by Lemma~\ref{minpring}. If $n\leq 4$ then $\mathcal{B}$ is $\CI$ by~\cite[Lemma~5.7]{KR2}. Thus, if $n=4$ then the statement of the lemma holds.

Let $n=5$. Suppose that $\mathcal{B}$ is indecomposable. Then the second part of Lemma~\ref{p45} implies $\mathcal{B}\cong \mathbb{Z}C_2\otimes \mathcal{B}^{\prime}$, where $\mathcal{B}^{\prime}$ is indecomposable $2$-$S$-ring over $C_2^4$. Since $\mathcal{B}$ is schurian by Lemma~\ref{schurian} and every $S$-ring over an elementary abelian group of rank at most~$4$ is $\CI$ by the above paragraph, we conclude that $\mathcal{B}$ is a $\CI$-$S$-ring by Lemma~\ref{cistar}.

Now suppose that $\mathcal{B}$ is decomposable, i.e. $\mathcal{B}$ is the nontrivial $S=U/L$-wreath product for some $\mathcal{B}$-section $S=U/L$. Clearly, $|G/L|\leq 16$. The $S$-ring $\mathcal{B}_{G/L}$ is a $2$-$S$-ring by Lemma~\ref{psection}. Since every $S$-ring over an elementary abelian group of rank at most~$4$ is $\CI$, $\mathcal{B}$ is a $\CI$-$S$-ring by Lemma~\ref{2cigwr}. The lemma is proved.
\end{proof}

\section{Proof of Theorem~\ref{main}}

Let $G=H \times P$, where $H \cong C_2^5$ and $P \cong  C_p$, where $p$ is a prime. These notations are valid until the end of the paper. If $p=2$ then $G$ is not a $\DCI$-group by~\cite{Now}. So in view of Lemma~\ref{cimin}, to prove Theorem~\ref{main}, it is sufficient to prove the following theorem.

\begin{theo}\label{main2}
Let $p$ be an odd prime and $K \in \Sup_2^{\rm min}(G_\r)$. Then $\mathcal{A}=V(K, G)$ is a $\CI$-$S$-ring.
\end{theo}

The proof of Proposition~\ref{main2} will be given in the end of the section. We start with the next lemma concerned with proper sections of $G$.

\begin{lemm}\label{subgroup}
Let $S$ be a section of $G$ such that $S\neq G$. Then every schurian $S$-ring over $S$ is a $\CI$-$S$-ring.
\end{lemm}

\begin{proof}
If $S\cong C_2^n$ for some $n\leq 5$ then we are done by Lemma~\ref{everyci}. Suppose that $S\cong C_2^n \times C_p$ for some $n\leq 4$. Then the statement of the lemma follows from~\cite[Remark~3.4]{KR2} whenever $n\leq 3$ and from~\cite[Remark~3.4,~Theorem~7.1]{KR2} whenever $n=4$. The lemma is proved.
\end{proof}

A key step towards the proof of Theorem~\ref{main2} is the following lemma.

\begin{lemm}\label{main3}
Let $\mathcal{A}$ be an $S$-ring over $G$ and $U$ an $\mathcal{A}$-subgroup with $U\geq P$. Suppose that $P$ is an $\mathcal{A}$-subgroup, $\mathcal{A}$ is the nontrivial $S$-wreath product, where $S=U/P$, $|S|=16$, and $\mathcal{A}_{G/P}$ is a $2$-$S$-ring. Then $\mathcal{A}$ is a $\CI$-$S$-ring.
\end{lemm}

\begin{lemm}\label{proof1}
In the conditions of Lemma~\ref{main3}, suppose that $S$ has a gwr-complement with respect to $\mathcal{A}_{G/P}$. Then $\mathcal{A}$ is a a $\CI$-$S$-ring.
\end{lemm}

\begin{proof}
The condition of the lemma implies that there exists an $\mathcal{A}_{G/P}$-subgroup $A$ such that $\mathcal{A}_{G/P}$ is the nontrivial $S/A$-wreath product. This means that $\mathcal{A}$ is the nontrivial $U/\pi^{-1}(A)$-wreath product, where $\pi: G\rightarrow G/P$ is the canonical epimorphism. Note that $|G/\pi^{-1}(A)|\leq 16$ and $\mathcal{A}_{G/\pi^{-1}(A)}\cong \mathcal{A}_{(G/P)/A}$ is a $2$-$S$-ring by Lemma~\ref{psection}. Therefore $\mathcal{A}$ is a $\CI$-$S$-ring by Lemma~\ref{subgroup} and Lemma~\ref{2cigwr}. The lemma is proved.
\end{proof}

\begin{lemm}\label{proof2}
In the conditions of Lemma~\ref{main3}, suppose that $S$ does not have a gwr-complement with respect to $\mathcal{A}_{G/P}$. Then 
$$|\aut_{G/P}(\mathcal{A}_{G/P})^S|=|\aut_{G/P}(\mathcal{A}_{G/P})|.$$
\end{lemm}

\begin{proof}
To prove the lemma it is sufficient to prove that the group
$$(\aut_{G/P}(\mathcal{A}_{G/P}))_S=\{\varphi\in \aut_{G/P}(\mathcal{A}_{G/P}):~\varphi^S=\id_S\}$$ is trivial. Let $\varphi\in (\aut_{G/P}(\mathcal{A}_{G/P}))_S$. Put $\mathcal{C}=\cyc(\langle \varphi \rangle, G/P)$. Clearly, $\langle \varphi \rangle\leq \aut(\mathcal{A}_{G/P})$. So from Eqs.~(1) and~(2) it follows that $\mathcal{C}\geq \mathcal{A}_{G/P}$. Lemma~\ref{pextension} yields that $\mathcal{C}$ is a $2$-$S$-ring. Since $\varphi^S=\id_S$, we conclude that $\O_\theta(\mathcal{C})\geq S$. 

If $\mathcal{C}\neq \mathbb{Z}(G/P)$ then $\O_\theta(\mathcal{C})=S$. Therefore $\mathcal{C}=\mathbb{Z}S\wr_{S/A} \mathbb{Z}((G/P)/A)$ for some $\mathcal{C}$-subgroup $A$ by Statement~1 of~\cite[Proposition~4.3]{KR1}. This implies that $\mathcal{A}_{G/P}=\mathcal{A}_S\wr_{S/A} \mathcal{A}_{((G/P)/A)}$ because $\mathcal{C}\geq \mathcal{A}_{G/P}$ and $S$ is both $\mathcal{A}_{G/P}$, $\mathcal{C}$-subgroup. We obtain a contradiction with the assumption of the lemma. Thus, $\mathcal{C}=\mathbb{Z}(G/P)$ and hence $\varphi$ is trivial. So the group $(\aut_{G/P}(\mathcal{A}_{G/P}))_S$ is trivial. The lemma is proved.
\end{proof}

\begin{proof}[Proof of Lemma~\ref{main3}.]

If $\mathcal{A}_{G/P}$ is indecomposable then $\mathcal{A}_{G/P}$ is normal by Lemma~\ref{p45}. So $\mathcal{A}$ is a $\CI$-$S$-ring by Lemma~\ref{subgroup} and Lemma~\ref{kerci}. Further we assume that $\mathcal{A}_{G/P}$ is decomposable. Due to Lemma~\ref{proof1}, we may assume also that 
$$S~\text{does not have a gwr-complement with respect to}~\mathcal{A}_{G/P}.~\eqno(6)$$ 
If there exists $X\in \mathcal{S}(\mathcal{A}_{G/P})$ outside $S$ with $|X|=1$ then $\mathcal{A}$ is a $\CI$-$S$-ring by Lemma~\ref{subgroup} and Lemma~\ref{citens}. So we may assume that 
$$\O_\theta(\mathcal{A}_{G/P})\leq S.~\eqno(7)$$ 
Note that $|\O_\theta(\mathcal{A}_{G/P})|>1$ by Statement~1 of Lemma~\ref{psring} and $|\O_\theta(\mathcal{A}_{G/P})|\leq 16$ by Eq.~(7). So $|\O_\theta(\mathcal{A}_{G/P})|\in \{2,4,8,16\}$. We divide the rest of the proof into four cases depending on $|\O_\theta(\mathcal{A}_{G/P})|$.
\\
\\
\noindent~\textbf{Case~1: $|\O_\theta(\mathcal{A}_{G/P})|=16$.} 
\\
\\
Due to Eq.~(7), we conclude that $\mathcal{A}_S=\mathbb{Z}S$. So $\mathcal{A}$ is a $\CI$-$S$-ring by Lemma~\ref{subgroup} and Lemma~\ref{trivial}.
\\
\\
\noindent~\textbf{Case~2: $|\O_\theta(\mathcal{A}_{G/P})|=8$.} 
\\
\\
Since $\mathcal{A}_{G/P}$ is decomposable, Lemma~\ref{p5cyclotom} implies that $\mathcal{A}_{G/P}$ is cyclotomic. The $S$-ring $\mathcal{A}_S$ is a $2$-$S$-ring by Lemma~\ref{psection}. In view of Eq.~(7), we obtain that $|\O_\theta(\mathcal{A}_{S})|=8$. So Statement~2 of~\cite[Proposition~4.3]{KR1} yields that the $S$-ring $\mathcal{A}_S$ is Cayley minimal. Thus, $\mathcal{A}$ is a $\CI$-$S$-ring by Lemma~\ref{subgroup} and Lemma~\ref{cicayleymin}.
\\
\\
\noindent~\textbf{Case~3: $|\O_\theta(\mathcal{A}_{G/P})|=4$.} 
\\
\\
In this case one of the statements of Lemma~\ref{p5decomp1} holds for $\mathcal{A}_{G/P}$. If  Statement~1 of Lemma~\ref{p5decomp1} holds for $\mathcal{A}_{G/P}$ then we obtain a contradiction with Eq.~(6). 

If Statement~2 of Lemma~\ref{p5decomp1} holds for $\mathcal{A}_{G/P}$ then $|\aut_{G/P}(\mathcal{A}_{G/P})|\geq |\aut_S(\mathcal{A}_S)|$. From Lemma~\ref{proof2} it follows that $|\aut_{G/P}(\mathcal{A}_{G/P})^S|=|\aut_{G/P}(\mathcal{A}_{G/P})|$ and hence
$$|\aut_{G/P}(\mathcal{A}_{G/P})^S|\geq |\aut_S(\mathcal{A}_S)|.$$
Since $\aut_{G/P}(\mathcal{A}_{G/P})^S\leq \aut_S(\mathcal{A}_S)$, we conclude that $\aut_{G/P}(\mathcal{A}_{G/P})^S=\aut_S(\mathcal{A}_S)$. Thus, $\mathcal{A}$ is a $\CI$-$S$-ring by Lemma~\ref{subgroup} and Lemma~\ref{cigwr}.

If Statement~3 of Lemma~\ref{p5decomp1} holds for $\mathcal{A}_{G/P}$ then $\mathcal{A}_{G/P}$ is normal. In this case $\mathcal{A}$ is a $\CI$-$S$-ring by Lemma~\ref{subgroup} and Lemma~\ref{kerci}.

Suppose that Statement~4 of Lemma~\ref{p5decomp1} holds for $\mathcal{A}_{G/P}$, i.e. there exists an $\mathcal{A}_{G/P}$-subgroup $A\leq \O_\theta(\mathcal{A}_{G/P})$ of order~$2$ and $X=\{x_1,x_2\}\in \mathcal{S}(\mathcal{A}_{G/P})$ such that $\mathcal{A}_{(G/P)/A}$ is normal and $A\neq \rad(X)$. Let $L=\pi^{-1}(A)$, where $\pi:G\rightarrow G/P$ is the canonical epimorphism, and $\mathcal{B}=V(N,G)$, where $N=\aut(\mathcal{A})_{G/L}G_{\r}$. 

Prove that $\mathcal{B}$ is a $\CI$-$S$-ring. Lemma~\ref{kerwreath} implies that $\mathcal{B}$ is the $S$-wreath product. From Eqs.~(1) and~(2) it follows that $\mathcal{B}\geq \mathcal{A}$.  So $\mathcal{B}_{G/P}\geq \mathcal{A}_{G/P}$ and hence $\mathcal{B}_{G/P}$ is a $2$-$S$-ring by Lemma~\ref{pextension}. We obtain that $\mathcal{B}$ and $U$ satisfy the conditions of Lemma~\ref{main3}.

One can see that $X$ is a $\mathcal{B}_{G/P}$-set and 
$$\O_{\theta}(\mathcal{B}_{G/P})\geq \O_{\theta}(\mathcal{A}_{G/P})~\eqno(8)$$
because $\mathcal{B}_{G/P}\geq \mathcal{A}_{G/P}$. The definition of $\mathcal{B}$ yields that every basic set of $\mathcal{B}$ is contained in an $L$-coset and hence every basic set of $\mathcal{B}_{G/P}$ is contained in an $A$-coset. Therefore 
$$\{x_1\},\{x_2\}\in \mathcal{S}(\mathcal{B}_{G/P})~\eqno(9)$$ 
because $X$ is a $\mathcal{B}_{G/P}$-set and $A\neq \rad(X)$. Now from Eqs.~(8) and~(9) it follows that
$$|\O_{\theta}(\mathcal{B}_{G/P})|\geq 8.~\eqno(10)$$

If $\mathcal{B}_{G/P}$ is indecomposable then $\mathcal{B}_{G/P}$ is normal by Lemma~\ref{p45} and hence $\mathcal{B}$ is $\CI$ by Lemma~\ref{subgroup} and Lemma~\ref{kerci}; if $S$ has a gwr-complement with respect to $\mathcal{B}_{G/P}$ then $\mathcal{B}$ is $\CI$ by Lemma~\ref{proof1}; if $\O_\theta(\mathcal{B}_{G/P})\nleq S$ then $\mathcal{B}$ is $\CI$ by Lemma~\ref{subgroup} and Lemma~\ref{citens}; otherwise $\mathcal{B}$ is $\CI$ by Eq.~(10) and one of the Cases~1 or~2. 

Clearly, $\mathcal{A}_{G/L}\cong\mathcal{A}_{(G/P)/A}$ and hence $\mathcal{A}_{G/L}$ is normal. Also $\mathcal{A}_{G/L}$ is $\CI$ by Lemma~\ref{subgroup}. The $S$-ring  $\mathcal{B}$ is $\CI$ by the above paragraph.  Thus, $\mathcal{A}$ is $\CI$ by Lemma~\ref{minnorm}.
\\
\\
\noindent~\textbf{Case~4: $|\O_\theta(\mathcal{A}_{G/P})|=2$.} 
\\
\\
Let $A=\O_\theta(\mathcal{A}_{G/P})$. Clearly, $A$ is the least $\mathcal{A}_{G/P}$-subgroup. If $|\rad(X)|>1$ for every $X\in \mathcal{S}(\mathcal{A}_{G/P})$ outside $S$ then $A\leq \rad(X)$ for every $X\in \mathcal{S}(\mathcal{A}_{G/P})$ outside $S$ and we obtain a contradiction with Eq.~(6). So there exists $X\in \mathcal{S}(\mathcal{A}_{G/P})$ outside $S$ with $|\rad(X)|=1$. From Eq.~(7) it follows that $|X|>1$. Lemma~\ref{p5decomp2} implies that $|X|=4$. The number $\lambda=|X \cap Ax|$ does not depend on $x\in X$ by Lemma~\ref{intersection}. If $\lambda=2$ then $A\leq \rad(X)$, a contradiction. Therefore
$$\lambda=1.~\eqno(11)$$ 

One of the statements of Lemma~\ref{p5decomp2} holds for $\mathcal{A}_{G/P}$. If Statement~1  of Lemma~\ref{p5decomp2} holds for $\mathcal{A}_{G/P}$ then there exists $Y\in \mathcal{S}(\mathcal{A}_{G/P})$ with $|Y|=16$ and $|\rad(Y)|=16$. Since $|S|=16$, we conclude that $Y$ lies outside $S$ and hence $Y=(G/P)\setminus S$. This means that $S$ is a gwr-complement to $S$ with respect to $\mathcal{A}_{G/P}$. However, this contradicts to Eq.~(6).

If Statement~2 of Lemma~\ref{p5decomp2} holds for $\mathcal{A}_{G/P}$ then $|\aut_{G/P}(\mathcal{A}_{G/P})|\geq |\aut_S(\mathcal{A}_S)|$. So Lemma~\ref{proof2} implies that $\aut_{G/P}(\mathcal{A}_{G/P})^S=\aut_S(\mathcal{A}_S)$. Therefore, $\mathcal{A}$ is $\CI$ by Lemma~\ref{subgroup} and Lemma~\ref{cigwr}

Suppose that Statement~3 of Lemma~\ref{p5decomp2} holds for $\mathcal{A}_{G/P}$, i.e. there exists an $\mathcal{A}_{G/P}$-subgroup $B$ such that $|B|\in\{2,4\}$ and $\mathcal{A}_{(G/P)/B}$ is normal. Let $L=\pi^{-1}(B)$, where $\pi:G\rightarrow G/P$ is the canonical epimorphism, and $\mathcal{B}=V(N,G)$, where $N=\aut(\mathcal{A})_{G/L}G_{\r}$. 

Prove that $\mathcal{B}$ is a $\CI$-$S$-ring. As in Case~3, $\mathcal{B}$ is the $S$-wreath product by Lemma~\ref{kerwreath} and $\mathcal{B}\geq \mathcal{A}$ by Eqs.~(1) and~(2). So $\mathcal{B}_{G/P}\geq \mathcal{A}_{G/P}$ and hence $\mathcal{B}_{G/P}$ is a $2$-$S$-ring by Lemma~\ref{pextension}. Therefore $\mathcal{B}$ and $U$ satisfy the conditions of Lemma~\ref{main3}.

Note that $X$ is a $\mathcal{B}_{G/P}$-set and Eq.~(8) holds because $\mathcal{B}_{G/P}\geq \mathcal{A}_{G/P}$. By the definition of $\mathcal{B}$, every basic set of $\mathcal{B}$ is contained in an $L$-coset and hence every basic set of $\mathcal{B}_{G/P}$ is contained in a $B$-coset. The set $X$ is a $\mathcal{B}_{G/P}$-set with $|X|=4$ and $|\rad(X)|=1$. So there exists $X_1\in \mathcal{S}(\mathcal{B}_{G/P})$ such that 
$$X_1\subset X~\text{and}~|X_1|\in\{1,2\}.$$ 

If $|X_1|=1$ then $X_1\subseteq \O_{\theta}(\mathcal{B}_{G/P})$. If $|X_1|=2$ then $X_1$ is a coset by a $\mathcal{B}_{G/P}$-subgroup $A_1$ of order~$2$. Clearly, $A_1\subseteq \O_{\theta}(\mathcal{B}_{G/P})$. In view of Eq.~(11), we have $A_1\neq A$. Thus, in both cases $\O_{\theta}(\mathcal{B}_{G/P})\nleq A$. Together with Eq.~(8) this implies that 
$$|\O_{\theta}(\mathcal{B}_{G/P})|\geq 4.~\eqno(12)$$

If $\mathcal{B}_{G/P}$ is indecomposable then $\mathcal{B}_{G/P}$ is normal by Lemma~\ref{p45} and hence $\mathcal{B}$ is $\CI$ by Lemma~\ref{subgroup} and Lemma~\ref{kerci}; if $S$ has a gwr-complement with respect to $\mathcal{B}_{G/P}$ then $\mathcal{B}$ is $\CI$ by Lemma~\ref{proof1}; if $\O_\theta(\mathcal{B}_{G/P})\nleq S$ then $\mathcal{B}$ is $\CI$ by Lemma~\ref{subgroup} and Lemma~\ref{citens}; otherwise $\mathcal{B}$ is $\CI$ by Eq.~(12) and one of the Cases~1, 2, or~3. 

The $S$-ring $\mathcal{A}_{G/L}$ is normal because it is isomorphic to $\mathcal{A}_{(G/P)/B}$. The $S$-rings $\mathcal{A}_{G/L}$ and $\mathcal{B}$ are $\CI$ by Lemma~\ref{subgroup} and the above paragraph respectively. Thus, $\mathcal{A}$ is $\CI$ by Lemma~\ref{minnorm}.

All cases were considered. The lemma is proved.
\end{proof}

\begin{proof}[Proof of Theorem~\ref{main2}.]

Let $H_1$ be a maximal $\mathcal{A}$-subgroup contained in $H$ and $P_1$ the least $\mathcal{A}$-subgroup containing $P$. 

\begin{lemm}\label{proof3}
If $H_1=H$ then $\mathcal{A}$ is a $\CI$-$S$-ring.
\end{lemm}

\begin{proof}
The $S$-ring $\mathcal{A}_{G/H}$ is a $p$-$S$-ring over $G/H\cong C_p$ by Lemma~\ref{minpring}. So $\mathcal{A}_{G/H}\cong \mathbb{Z}C_p$. Clearly, $G=HP_1$. Therefore $\mathcal{A}=\mathcal{A}_{H} \star \mathcal{A}_{P_1}$ by Lemma~\ref{nonpower3}. Since $H$ and $P_1/(H\cap P_1)$ are proper sections of $G$, the $S$-rings $\mathcal{A}_H$ and $\mathcal{A}_{P_1/(H\cap P_1)}$ are $\CI$-$S$-ring by Lemma~\ref{subgroup}. Thus, $\mathcal{A}$ is a $\CI$-$S$-ring by Lemma~\ref{cistar}. The lemma is proved.
\end{proof}
 
\begin{lemm}\label{proof4}
If $H_1<H$ and $H_1P_1=G$ then $\mathcal{A}$ is a $\CI$-$S$-ring.
\end{lemm}

\begin{proof}
Since $H_1 \neq (H_1P_1)_{p'}=H$, Lemma~\ref{nonpower2} implies that $\mathcal{A}=\mathcal{A}_{H_1} \star \mathcal{A}_{P_1}$. The $S$-rings $\mathcal{A}_{H_1}$ and $\mathcal{A}_{P_1/(H\cap P_1)}$ are $\CI$-$S$-ring by Lemma~\ref{subgroup} because $H_1$ and $P_1/(H_1\cap P_1)$ are proper sections of $G$. Therefore $\mathcal{A}$ is a $\CI$-$S$-ring by Lemma~\ref{cistar}. The lemma is proved.
\end{proof}

In view of Lemma~\ref{proof3}, we may assume that $H_1<H$. Then one of the statements of Lemma~\ref{nonpower4} holds for $\mathcal{A}$. If Statement~1 of Lemma~\ref{nonpower4} holds for $\mathcal{A}$ then 
$$\mathcal{A}=\mathcal{A}_{H_1}\wr \mathcal{A}_{G/H_1},$$ 
where $\rk(\mathcal{A}_{G/H_1})=2$. If $H_1$ is trivial then $\rk(\mathcal{A})=2$ and obviously $\mathcal{A}$ is a $\CI$-$S$-ring. If $H_1$ is nontrivial then $\mathcal{A}$ is a $\CI$-$S$-ring by Lemma~\ref{subgroup} and Lemma~\ref{trivial}.

Assume that Statement~2 of Lemma~\ref{nonpower4} holds for $\mathcal{A}$, i.e.
$$\mathcal{A}=\mathcal{A}_U\wr_S \mathcal{A}_{G/P_1},$$ 
where $U=H_1P_1$, $S=U/P_1$, and $P_1<G$. In view of Lemma~\ref{proof4}, we may assume that $H_1P_1<G$, i.e. $\mathcal{A}$ is the nontrivial $S$-wreath product. The group $G/P_1$ is a $2$-group of order at most~$32$ because $P_1\geq P$. So Lemma~\ref{minpring} implies that $\mathcal{A}_{G/P_1}$ is a $2$-$S$-ring. If $|G/P_1|\leq 16$ then $\mathcal{A}$ is a $\CI$-$S$-ring by Lemma~\ref{subgroup} and Lemma~\ref{2cigwr}. So we may assume that $|G/P_1|=32$. Clearly, in this case 
$$P_1=P.$$ 
In view of Statement~2 of Lemma~\ref{psring}, we may assume also that 
$$|S|=16.$$ 
Indeed, if $|S|<16$ then $S$ is contained in an $\mathcal{A}_{G/P}$-subgroup $S^{\prime}$ of order $16$ by Statement~2 of Lemma~\ref{psring}. Clearly, $\mathcal{A}=\mathcal{A}_{U^{\prime}}\wr_{S^{\prime}}\mathcal{A}_{G/P}$, where $U^{\prime}=\pi^{-1}(S^{\prime})$ and $\pi:G\rightarrow G/P$ is the canonical epimorphism. Replacing $S$ by $S^{\prime}$, we obtain the required. 

Now all conditions of Lemma~\ref{main3} hold for $\mathcal{A}$ and $U$. Thus, $\mathcal{A}$ is a $\CI$-$S$-ring by Lemma~\ref{main3}. The theorem is proved.
\end{proof}

\end{document}